\documentclass{article}
\usepackage{authblk}
\usepackage{a4wide}
\usepackage{amsthm}
\usepackage{graphicx}
\usepackage{amssymb}
\usepackage{amsmath}
\usepackage{ascmac}
\usepackage{setspace}
\usepackage{float}
\usepackage{algorithm}
\usepackage{algorithmic}
\usepackage{setspace}
\usepackage{tikz-cd}
\newtheorem{Theorem}[equation]{Theorem}
\newtheorem{Corollary}[equation]{Corollary}
\newtheorem{Lemma}[equation]{Lemma}
\newtheorem{Proposition}[equation]{Proposition}

\theoremstyle{definition}
\newtheorem{Definition}[equation]{Definition}

\theoremstyle{remark}
\newtheorem{Remark}[equation]{Remark}
\numberwithin{equation}{section}

\DeclareMathOperator{\id}{id}

\DeclareMathOperator{\ad}{ad}

\newcommand{\ve}{\varepsilon}

\allowdisplaybreaks
\begin{document}
\title{A homomorphism from the affine Yangian $Y_{\hbar,\ve}(\widehat{\mathfrak{sl}}(n))$ to the affine Yangian $Y_{\hbar,\ve}(\widehat{\mathfrak{sl}}(n+1))$}
\author{Mamoru Ueda\thanks{mueda@ualberta.ca}}
\affil{Department of Mathematical and Statistical Sciences, University of Alberta, 11324 89 Ave NW, Edmonton, AB T6G 2J5, Canada}
\date{}
\maketitle
\begin{abstract}
We construct a homomorphism from the affine Yangian $Y_{\hbar,\ve}(\widehat{\mathfrak{sl}}(n))$ to the standard degreewise completion of the affine Yangian $Y_{\hbar,\ve}(\widehat{\mathfrak{sl}}(n+1))$. We also give the relationship between this homomorphism and the one from the affine Yangian $Y_{\hbar,\ve}(\widehat{\mathfrak{sl}}(n))$ to the universal enveloping algebra of the rectangular algebra $\mathcal{W}^k(\mathfrak{gl}(2n),(2^n))$ constructed by the author \cite{U4}.
\end{abstract}
\section{Introduction}
Drinfeld (\cite{D1}, \cite{D2}) introduced the Yangian associated with a finite dimensional simple Lie algebra $\mathfrak{g}$ in order to solve the Yang-Baxter equation. The Yangian is a quantum group which is a deformation of the current algebra $\mathfrak{g}\otimes\mathbb{C}[z]$. The Yangian of type $A$ has several presentations: the RTT presentation, the current presentation, and so on. 
By using the current presentation of the Yangian, we can extend the definition of the Yangian $Y_\hbar(\mathfrak{g})$ to a symmetrizable Kac-Moody Lie algebra $\mathfrak{g}$. For general symmetrizable Kac-Moody Lie algebra $\mathfrak{g}$, it has not been resolved whether the Yangian $Y_\hbar(\mathfrak{g})$ has a quantum group structure or not. However, in the case that $\mathfrak{g}$ is of affine type, this problem has been affirmatively resolved (\cite{BL}, \cite{GNW}, and \cite{U1}). 

Recently, relationships between affine Yangians and $W$-algebras have been studied. 
A $W$-algebra $\mathcal{W}^k(\mathfrak{g},f)$ is a vertex algebra associated with a finite dimensional reductive Lie algebra $\mathfrak{g}$ and a nilpotent element $f\in\mathfrak{g}$. It appeared in the study of two dimensional conformal field theories (\cite{Z}). 
We call a $W$-algebra associated with $\mathfrak{gl}(n)$ (resp. $\mathfrak{gl}(ln)$) and a principal nilpotent element (resp. a nilpotent element of type of $(l^n)$) a principal (resp. rectangular) $W$-algebra. The AGT (Alday-Gaiotto-Tachikawa) conjecture suggests that there exists a representation of the principal $W$-algebra of type $A$ on the equivariant homology space of the moduli space of $U(r)$-instantons. Schiffmann and Vasserot \cite{SV} gave this representation by using an action of the affine Yangian associated with $\widehat{\mathfrak{gl}}(1)$ on this homology space.

In the rectangular case, the author \cite{U4} constructed a surjective homomorphism from the Guay's affine Yangian (\cite{Gu2} and \cite{Gu1}) to the universal enveloping algebra of a rectangular $W$-algebra of type $A$:
\begin{equation*}
\Phi^n\colon Y_{\hbar,\ve}(\widehat{\mathfrak{sl}}(n))\to\mathcal{U}(\mathcal{W}^k(\mathfrak{gl}(ln),(l^n))).
\end{equation*}
The Guay's affine Yangian $Y_{\hbar,\ve}(\widehat{\mathfrak{sl}}(n))$ is a 2-parameter affine Yangian associated with $\widehat{\mathfrak{sl}}(n)$. The Guay's affine Yangian has a quantum group structure and is the deformation of the universal enveloping algebra of the central extension of $\mathfrak{sl}(n)[u^{\pm1},v]$. It is known that the Guay's affine Yangian has a representation on the equivariant homology space of affine Laumon spaces (\cite{FFNR} and \cite{FT}). Similarly to principal $W$-algebras, we expect that we can construct geometric representations of rectangular $W$-algebras by using the homomorphism $\Phi^n$. 

In non-rectangular cases, it is conjectured in Creutzig-Diaconescu-Ma \cite{CE} that a geometric representation of an iterated $W$-algebra of type $A$ on the equivariant homology space of the affine Laumon space will be given through an action of an affine shifted Yangian constructed in \cite{FT}.
Based on this conjecture, we can expect that there exists a non-trivial homomorphism from the affine shifted Yangian to an iterated $W$-algebra associated with any nilpotent element. However, tackling this issue is very difficult and has not been resolved.

In \cite{U7}, we constructed a homomorphism from the Guay's affine Yangian $Y_{\hbar,\ve}(\widehat{\mathfrak{sl}}(n))$ to the universal enveloping algebra of a non-rectangular $W$-algebra.
This is the affine version of De Sole-Kac-Valeri \cite{DKV}. In \cite{DKV}, De Sole, Kac and Valeri constructed a homomorphism from the finite Yangian of type $A$ to a finite $W$-algebras of type $A$ by using the Lax operator. The homomorphism in De Sole-Kac-Valeri \cite{DKV} is a restriction of the one from the shifted Yangian to a finite $W$-algebra given by Brundan-Kleshchev in \cite{BK}. We expect that we can extend the homomorphism given in \cite{U7} to the affine shifted Yangian and this extended homomorphism coincides with the homomorphism conjectured in Creutzig-Diaconescu-Ma \cite{CE}.

For this purpose, it is helpful to give a homomorphism from the affine Yangian to the affine shifted Yangian. As a first step, in this article, we construct a homomorphism 
\begin{equation*}
\Psi\colon Y_{\hbar,\ve}(\widehat{\mathfrak{sl}}(n))\to \widetilde{Y}_{\hbar,\ve}(\widehat{\mathfrak{sl}}(n+1)),
\end{equation*}
where $\widetilde{Y}_{\hbar,\ve}(\widehat{\mathfrak{sl}}(n+1))$ is the standard degreewise completion of $Y_{\hbar,\ve}(\widehat{\mathfrak{sl}}(n+1))$.
In finite setting, by using the RTT presentation, there exists a natural embedding from the Yangian associated with $\mathfrak{gl}(n)$ to the one associated with $\mathfrak{gl}(n+1)$. However, the Guay's affine Yangian $Y_{\hbar,\ve}(\widehat{\mathfrak{sl}}(n))$ does not have the RTT presentation.
For the construction of $\Psi$, we use the finite presentation called the {\it minimalistic presentation} given by Guay-Nakajima-Wendlandt \cite{GNW}.

As for one of the rectangular cases, we can construct a relationship with $\Psi$ and $\Phi^n$. There exists an embedding
\begin{equation*}
\iota\colon  \mathcal{W}^{k+1}(\mathfrak{gl}(2n),(2^n))\to(\mathcal{W}^k(\mathfrak{gl}(2n+2),(2^{n+1})).
\end{equation*}
 In the last section of this article, we show the following relation:
\begin{equation*}
\Phi^{n+1}\circ\Psi=\iota\circ\Phi^n.
\end{equation*}
We expect that the similar relation holds in the non-rectangular case.
\section{Affine Yangian}
Let us recall the definition of the affine Yangian of type $A$ (Definition~3.2 in \cite{Gu2} and Definition~2.3 of \cite{Gu1}). Here after, we identify $\{0,1,2,\cdots,n-1\}$ with $\mathbb{Z}/n\mathbb{Z}$.
\begin{Definition}\label{Def}
Suppose that $n\geq 3$. The affine Yangian $Y_{\ve_1,\ve_2}(\widehat{\mathfrak{sl}}(n))$ is the associative algebra over  generated by $x_{i,r}^{+}, x_{i,r}^{-}, h_{i,r}$ $(i \in \{0,1,\cdots, n-1\}, r \in \mathbb{Z}_{\geq 0})$ with two parameters $\ve_1, \ve_2 \in \mathbb{C}$ subject to the following defining relations:
\begin{gather}
	[h_{i,r}, h_{j,s}] = 0, \label{eq1.1}\\
	[x_{i,r}^{+}, x_{j,s}^{-}] = \delta_{i,j} h_{i, r+s}, \label{eq1.2}\\
	[h_{i,0}, x_{j,r}^{\pm}] = \pm a_{i,j} x_{j,r}^{\pm},\label{eq1.3}\\
	[h_{i, r+1}, x_{j, s}^{\pm}] - [h_{i, r}, x_{j, s+1}^{\pm}] 
	= \pm a_{i,j} \dfrac{\varepsilon_1 + \varepsilon_2}{2} \{h_{i, r}, x_{j, s}^{\pm}\} 
	- m_{i,j} \dfrac{\varepsilon_1 - \varepsilon_2}{2} [h_{i, r}, x_{j, s}^{\pm}],\label{eq1.4}\\
	[x_{i, r+1}^{\pm}, x_{j, s}^{\pm}] - [x_{i, r}^{\pm}, x_{j, s+1}^{\pm}] 
	= \pm a_{ij}\dfrac{\varepsilon_1 + \varepsilon_2}{2} \{x_{i, r}^{\pm}, x_{j, s}^{\pm}\} 
	- m_{ij} \dfrac{\varepsilon_1 - \varepsilon_2}{2} [x_{i, r}^{\pm}, x_{j, s}^{\pm}],\label{eq1.5}\\
	\sum_{w \in \mathfrak{S}_{1 + |a_{ij}|}}[x_{i,r_{w(1)}}^{\pm}, [x_{i,r_{w(2)}}^{\pm}, \dots, [x_{i,r_{w(1 + |a_{ij}|)}}^{\pm}, x_{j,s}^{\pm}]\dots]] = 0\  \text{ if }i \neq j,\label{eq1.6}
\end{gather}
where $\{X,Y\}=XY+YX$ and
\begin{equation*}
a_{i,j} =\begin{cases}
2&\text{if } i=j, \\
-1&\text{if }j=i\pm 1,\\
0&\text{otherwise,}
	\end{cases}
	 m_{i,j}=
	\begin{cases}
-1 &\text{if } i=j - 1,\\
1 &\text{if } i=j + 1,\\
0&\text{otherwise}.
	\end{cases}
\end{equation*}
\end{Definition}
Guay-Nakajima-Wendland \cite{GNW} gave the new presentation of the affine Yangian $Y_{\ve_1,\ve_2}(\widehat{\mathfrak{sl}}(n))$ whose generators are
\begin{equation*}
\{h_{i,r},x^\pm_{i,r}\mid 0\leq i\leq n-1,r\in\mathbb{Z}\}.
\end{equation*}
\begin{Proposition}\label{Prop32}
Suppose that $n\geq3$. The affine Yangian $Y_{\ve_1,\ve_2}(\widehat{\mathfrak{sl}}(n))$ is isomorphic to the associative algebra $Y_{\hbar,\ve}(\widehat{\mathfrak{sl}}(n))$ generated by $X_{i,r}^{+}, X_{i,r}^{-}, H_{i,r}$ $(i \in \{0,1,\cdots, n-1\}, r = 0,1)$ subject to the following defining relations:
\begin{gather}
[H_{i,r}, H_{j,s}] = 0,\label{Eq2.1}\\
[X_{i,0}^{+}, X_{j,0}^{-}] = \delta_{ij} H_{i, 0},\label{Eq2.2}\\
[X_{i,1}^{+}, X_{j,0}^{-}] = \delta_{ij} H_{i, 1} = [X_{i,0}^{+}, X_{j,1}^{-}],\label{Eq2.3}\\
[H_{i,0}, X_{j,r}^{\pm}] = \pm a_{ij} X_{j,r}^{\pm},\label{Eq2.4}\\
[\tilde{H}_{i,1}, X_{j,0}^{\pm}] = \pm a_{ij}\left(X_{j,1}^{\pm}\right),\text{ if }(i,j)\neq(0,n-1),(n-1,0),\label{Eq2.5}\\
[\tilde{H}_{0,1}, X_{n-1,0}^{\pm}] = \mp \left(X_{n-1,1}^{\pm}+(\ve+\dfrac{n}{2}\hbar) X_{n-1, 0}^{\pm}\right),\label{Eq2.6}\\
[\tilde{H}_{n-1,1}, X_{0,0}^{\pm}] = \mp \left(X_{0,1}^{\pm}-(\ve+\dfrac{n}{2}\hbar) X_{0, 0}^{\pm}\right),\label{Eq2.7}\\
[X_{i, 1}^{\pm}, X_{j, 0}^{\pm}] - [X_{i, 0}^{\pm}, X_{j, 1}^{\pm}] = \pm a_{ij}\dfrac{\hbar}{2} \{X_{i, 0}^{\pm}, X_{j, 0}^{\pm}\}\text{ if }(i,j)\neq(0,n-1),(n-1,0),\label{Eq2.8}\\
[X_{0, 1}^{\pm}, X_{n-1, 0}^{\pm}] - [X_{0, 0}^{\pm}, X_{n-1, 1}^{\pm}]= \mp\dfrac{\hbar}{2} \{X_{0, 0}^{\pm}, X_{n-1, 0}^{\pm}\} + (\ve+\dfrac{n}{2}\hbar) [X_{0, 0}^{\pm}, X_{n-1, 0}^{\pm}],\label{Eq2.9}\\
(\ad X_{i,0}^{\pm})^{1+|a_{i,j}|} (X_{j,0}^{\pm})= 0 \ \text{ if }i \neq j, \label{Eq2.10}
\end{gather}
where $\tilde{H}_{i,1}=H_{i,1}-\dfrac{\hbar}{2}H_{i,0}^2$, $\hbar=\ve_1+\ve_2$ and $\ve=-n\ve_1$.
\end{Proposition}
Proposition~\ref{Prop32} is a little different from the presentation given by Guay-Nakajima-Wendland \cite{GNW}.
The isomorphism $\Xi\colon Y_{\ve_1,\ve_2}(\widehat{\mathfrak{sl}}(n))\to Y_{\hbar,\ve}(\widehat{\mathfrak{sl}}(n))$ is given by
\begin{gather*}
\Xi(h_{i,0})=H_{i,0},\quad\Xi(x^\pm_{i,0})=X^\pm_{i,0},\\
\Xi(h_{i,1})=\begin{cases}
H_{0,1}&\text{ if $i=0$},\\
H_{i,1}-\dfrac{i}{2}(\ve_1-\ve_2)H_{i,0}&\text{ if $i\neq0$}.
\end{cases}
\end{gather*}
Thus, Proposition~\ref{Prop32} is derived from Guay-Nakajima-Wendland \cite{GNW}. By this corresponding, we find that
\begin{gather}
[X^\pm_{i,r},X^\pm_{j,s}]=0\text{ if }|i-j|>1,\label{gather1}\\
[X^\pm_{i,1},[X^\pm_{i,0},X^\pm_{j+1,r}]]+[X^\pm_{i,0},[X^\pm_{i,1},X^\pm_{j+1,r}]]=0.\label{gather2}
\end{gather}
\begin{Remark}
In \cite{U5}, the author gave the similar presentation for the affine super Yangian. We note that (2.29) in \cite{U5} contains a typo. We should replace (2.29) with
\begin{equation*}
[X_{0, 1}^{\pm}, X_{n-1, 0}^{\pm}] - [X_{0, 0}^{\pm}, X_{n-1, 1}^{\pm}]= \mp(-1)^{p(m+n)}\dfrac{\hbar}{2} \{X_{0, 0}^{\pm}, X_{n-1, 0}^{\pm}\} - (\ve+\dfrac{n}{2}\hbar) [X_{0, 0}^{\pm}, X_{n-1, 0}^{\pm}].
\end{equation*}
We also note that $\ve=-n\ve_2$ in \cite{U5}. This makes the difference between \eqref{Eq2.6}, \eqref{Eq2.7} and \eqref{Eq2.9} in Proposition~\ref{Prop32} and (2.26), (2.27) and (2.29) in \cite{U5}.
\end{Remark}
By the definition of the affine Yangian $Y_{\hbar,\ve}(\widehat{\mathfrak{sl}}(n))$, we find that there exists a natural homomorphism from the universal enveloping algebra of $\widehat{\mathfrak{sl}}(n)$ to $Y_{\hbar,\ve}(\widehat{\mathfrak{sl}}(n))$. In order to simplify the notation, we denote the image of $x\in U(\widehat{\mathfrak{sl}}(n))$ by $x$.

We take one completion of $Y_{\hbar,\ve}(\widehat{\mathfrak{sl}}(n))$. We set the degree of $Y_{\hbar,\ve}(\widehat{\mathfrak{sl}}(n))$ by
\begin{equation*}
\text{deg}(H_{i,r})=0,\ \text{deg}(X^\pm_{i,r})=\begin{cases}
\pm1&\text{ if }i=0,\\
0&\text{ if }i\neq0.
\end{cases}
\end{equation*}
We denote the standard degreewise completion of $Y_{\hbar,\ve}(\widehat{\mathfrak{sl}}(n))$ by $\widetilde{Y}_{\hbar,\ve}(\widehat{\mathfrak{sl}}(n))$. Let us set $A_i\in\widetilde{Y}_{\hbar,\ve}(\widehat{\mathfrak{sl}}(n))$ as
\begin{align*}
A_i&=\dfrac{\hbar}{2}\sum_{\substack{s\geq0\\u>v}}\limits E_{u,v}t^{-s}[E_{i,i},E_{v,u}t^s]+\dfrac{\hbar}{2}\sum_{\substack{s\geq0\\u<v}}\limits E_{u,v}t^{-s-1}[E_{i,i},E_{v,u}t^{s+1}]\\
&=\dfrac{\hbar}{2}\sum_{\substack{s\geq0\\u>i}}\limits E_{u,i}t^{-s}E_{i,u}t^s-\dfrac{\hbar}{2}\sum_{\substack{s\geq0\\i>v}}\limits E_{i,v}t^{-s}E_{v,i}t^s\\
&\quad+\dfrac{\hbar}{2}\sum_{\substack{s\geq0\\u>i}}\limits E_{u,i}t^{-s-1}E_{i,u}t^{s+1}-\dfrac{\hbar}{2}\sum_{\substack{s\geq0\\i>v}}\limits E_{i,v}t^{-s-1}E_{v,i}t^{s+1},
\end{align*}
where $E_{i,j}$ is a matrix unit whose $(a,b)$ component is $\delta_{a,i}\delta_{b,j}$
Similarly to Section~3 in \cite{GNW}, we define
\begin{align*}
J(h_i)&=\widetilde{H}_{i,1}-A_i+A_{i+1}\in \widetilde{Y}_{\hbar,\ve}(\widehat{\mathfrak{sl}}(n))
\end{align*}
We also set $J(x^\pm_i)=\pm\dfrac{1}{2}[J(h_i),x^\pm_i]$.
Guay-Nakajima-Wendland \cite{GNW} defined the automorphism of $Y_{\hbar,\ve}(\widehat{\mathfrak{sl}}(n))$ by
\begin{equation*}
\tau_i=\exp(\ad(x^+_{i,0}))\exp(-\ad(x^-_{i,0}))\exp(\ad(x^+_{i,0})).
\end{equation*}
Let $\alpha$ be a positive real root. By definition of the Weyl algebra, there is an element $w$ of the Weyl group of $\widehat{\mathfrak{sl}}(n)$ and a simple root $\alpha_j$ such that $\alpha=w\alpha_j$. Then we define a corresponding root vector by
\begin{equation*}
x^\pm_\alpha=\tau_{i_1}\tau_{i_2}\cdots\tau_{i_{p-1}}(x^\pm_{j}),
\end{equation*}
where $w =s_{i_1}s_{i_2}\cdots s_{i_{p-1}}$ is a reduced expression of $w$.
We can define $J(x^\pm_\alpha)$ as
\begin{equation*}
J(x^\pm_\alpha)=\tau_{i_1}\tau_{i_2}\cdots\tau_{i_{p-1}}J(x^\pm_{j}).
\end{equation*} 
\begin{Lemma}[Proposition 3.21 in \cite{GNW}]\label{J}
The following relation holds:
\begin{equation*}
[J(h_i),x^\pm_\alpha]=\pm(\alpha_i,\alpha)J(x_\alpha^\pm)\pm c_{\alpha,i}x_\alpha^\pm
\end{equation*}
for some $c_{\alpha,i}\in\mathbb{C}$.
\end{Lemma}
\section{A homomorphism from the affine Yangian $Y_{\hbar,\ve}(\widehat{\mathfrak{sl}}(n))$ to the affine Yangian $Y_{\hbar,\ve}(\widehat{\mathfrak{sl}}(n+1))$}
In this section, we will construct a homomorphism from the affine Yangian $Y_{\hbar,\ve}(\widehat{\mathfrak{sl}}(n))$ to the degreewise completion of the affine Yangian $Y_{\hbar,\ve}(\widehat{\mathfrak{sl}}(n+1))$.
\begin{Theorem}\label{Main}
There exists an algebra homomorphism
\begin{equation*}
\Psi\colon Y_{\hbar,\ve}(\widehat{\mathfrak{sl}}(n))\to \widetilde{Y}_{\hbar,\ve}(\widehat{\mathfrak{sl}}(n+1))
\end{equation*}
determined by
\begin{gather*}
\Psi(H_{i,0})=\begin{cases}
H_{0,0}+H_{n,0}&\text{ if }i=0,\\
H_{i,0}&\text{ if }i\neq 0,
\end{cases}\\
\Psi(X^+_{i,0})=\begin{cases}
E_{n,1}t&\text{ if }i=0,\\
E_{i,i+1}&\text{ if }i\neq 0,
\end{cases}\ 
\Psi(X^-_{i,0})=\begin{cases}
E_{1,n}t^{-1}&\text{ if }i=0,\\
E_{i+1,i}&\text{ if }i\neq 0,
\end{cases}
\end{gather*}
and
\begin{align*}
\Psi(H_{i,1})&= H_{i,1}-\hbar\displaystyle\sum_{s \geq 0} \limits E_{i,n+1}t^{-s-1} E_{n+1,i}t^{s+1}+\hbar\displaystyle\sum_{s \geq 0}\limits E_{i+1,n+1}t^{-s-1} E_{n+1,i+1}t^{s+1},\\
\Psi(X^+_{i,1})&=X^+_{i,1}-\hbar\displaystyle\sum_{s \geq 0}\limits E_{i,n+1}t^{-s-1} E_{n+1,i+1}t^{s+1},\\
\Psi(X^-_{i,1})&=X^-_{i,1}-\hbar\displaystyle\sum_{s \geq 0}\limits E_{i+1,n+1}t^{-s-1} E_{n+1,i}t^{s+1}
\end{align*}
for $i\neq 0$. In particular, we have
\begin{align*}
\Psi(\widetilde{H}_{i,1})&= \widetilde{H}_{i,1}-\hbar\displaystyle\sum_{s \geq 0} \limits E_{i,n+1}t^{-s-1} E_{n+1,i}t^{s+1}+\hbar\displaystyle\sum_{s \geq 0}\limits E_{i+1,n+1}t^{-s-1} E_{n+1,i+1}t^{s+1}
\end{align*}
for $i\neq 0$.
\end{Theorem}
By Theorem~\ref{Main}, we can easily compute $\Psi(X^+_{0,1})$ and $\Psi(H_{0,1})$.
\begin{Corollary}\label{Cor}
The following equations hold:
\begin{align*}
\Psi(X^+_{0,1})&=[X^+_{n,0},X^+_{0,1}]-\hbar\displaystyle\sum_{s \geq 0} \limits E_{n,n+1}t^{-s} E_{n+1,1}t^{s+1},\\
\Psi(X^-_{0,1})&=[X^-_{0,1},X^-_{n,0}]-\hbar\displaystyle\sum_{s \geq 0} \limits E_{1,n+1}t^{-s-1} E_{n+1,n}t^{s},\\
\Psi(H_{0,1})&=H_{0,1}+H_{n,1}+(\ve+\dfrac{\hbar}{2}n)H_{n,0}+\hbar H_{n,0}H_{0,0}\\
&\quad-\hbar\displaystyle\sum_{s \geq 0} \limits E_{n,n+1}t^{-s-1} E_{n+1,n}t^{s+1}+\hbar\displaystyle\sum_{s \geq 0} \limits E_{1,n+1}t^{-s-1} E_{n+1,1}t^{s+1}.
\end{align*}
In particular, we obtain 
\begin{align*}
\Psi(\widetilde{H}_{0,1})&=\widetilde{H}_{0,1}+\widetilde{H}_{n,1}+(\ve+\dfrac{\hbar}{2}n)H_{n,0}\\
&\quad-\hbar\displaystyle\sum_{s \geq 0} \limits E_{n,n+1}t^{-s-1} E_{n+1,n}t^{s+1}+\hbar\displaystyle\sum_{s \geq 0} \limits E_{1,n+1}t^{-s-1} E_{n+1,1}t^{s+1}.
\end{align*}
\end{Corollary}
\begin{proof}
By the definition of $\Psi$ and \eqref{Eq2.5}, we have
\begin{align*}
\Psi(X^+_{0,1})&=-[\Psi(\widetilde{H}_{1,1}),\Psi(X^+_{0,0})]\\
&=-[\widetilde{H}_{1,1},[X^+_{n,0},X^+_{0,0}]]+\hbar[\displaystyle\sum_{s \geq 0} \limits E_{1,n+1}t^{-s-1} E_{n+1,1}t^{s+1},E_{n,1}t]\\
&\quad-\hbar[\displaystyle\sum_{s \geq 0}\limits E_{2,n+1}t^{-s-1} E_{n+1,2}t^{s+1},E_{n,1}t]\\
&=[X^+_{n,0},X^+_{0,1}]-\hbar\displaystyle\sum_{s \geq 0} \limits E_{n,n+1}t^{-s} E_{n+1,1}t^{s+1},
\end{align*}
where the second equality is due to $E_{n,1}t=[X^+_{n,0},X^+_{0,0}]$.
Similarly to $\Psi(X^+_{0,1})$, we can compute $\Psi(X^-_{0,1})$.

By \eqref{Eq2.3}, we obtain
\begin{align}
\Psi(H_{0,1})&=[\Psi(X^+_{0,1}),\Psi(X^-_{0,0})]\nonumber\\
&=[[X^+_{n,0},X^+_{0,1}],[X^-_{0,0},X^-_{n,0}]]-[\hbar\displaystyle\sum_{s \geq 0} \limits E_{n,n+1}t^{-s} E_{n+1,1}t^{s+1},E_{1,n}t^{-1}]\nonumber\\
&=[[X^+_{n,0},H_{0,1}],X^-_{n,0}]+[X^-_{0,0},[H_{n,0},X^+_{0,1}]]-[\hbar\displaystyle\sum_{s \geq 0} \limits E_{n,n+1}t^{-s} E_{n+1,1}t^{s+1},E_{1,n}t^{-1}],\label{al-0}
\end{align}
where the last equality is due to \eqref{Eq2.2} and \eqref{Eq2.3}.
By \eqref{Eq2.3} and \eqref{Eq2.4}, we have
\begin{align}
[X^-_{0,0},[H_{n,0},X^+_{0,1}]]&=-[X^-_{0,0},X^+_{0,1}]=H_{0,1}.\label{al-1}
\end{align}
By \eqref{Eq2.3} and \eqref{Eq2.5}-\eqref{Eq2.7}, we have
\begin{align}
&\quad[[X^+_{n,0},H_{0,1}],X^-_{n,0}]\nonumber\\
&=-[[\widetilde{H}_{0,1}+\dfrac{\hbar}{2}H_{0,0}^2,X^+_{n,0}],X^-_{n,0}]\nonumber\\
&=[X^+_{n,1}+(\ve+\dfrac{\hbar}{2}(n+1))X^+_{n,0},X^-_{n,0}]+\dfrac{\hbar}{2}[\{H_{0,0},X^+_{n,0}\},X^-_{n,0}]\nonumber\\
&=H_{n,1}+(\ve+\dfrac{\hbar}{2}(n+1))H_{n,0}+\dfrac{\hbar}{2}\{X^-_{n,0},X^+_{n,0}\}+\dfrac{\hbar}{2}\{H_{0,0},H_{n,0}\}\nonumber\\
&=H_{n,1}+(\ve+\dfrac{\hbar}{2}(n+1))H_{n,0}+\hbar X^+_{n,0}X^-_{n,0}-\dfrac{\hbar}{2}H_{n,0}+\hbar H_{0,0},H_{n,0},\label{al-1.5}
\end{align}
where the second equality is due to \eqref{Eq2.6} and \eqref{Eq2.4}, the third equality is due to \eqref{Eq2.2} and \eqref{Eq2.3} and the last equality is due to \eqref{Eq2.1} and the relation
\begin{align*}
\dfrac{\hbar}{2}\{X^-_{n,0},X^+_{n,0}\}&=\hbar X^+_{n,0}X^-_{n,0}-\dfrac{\hbar}{2}H_{n,0}.
\end{align*}
By a direct computation, we obtain
\begin{align}
&\quad-[\hbar\displaystyle\sum_{s \geq 0} \limits E_{n,n+1}t^{-s} E_{n+1,1}t^{s+1},E_{1,n}t^{-1}]\nonumber\\
&=-\hbar\displaystyle\sum_{s \geq 0} \limits E_{n,n+1}t^{-s} E_{n+1,n}t^{s}+\hbar\displaystyle\sum_{s \geq 0} \limits E_{1,n+1}t^{-s-1} E_{n+1,1}t^{s+1}.\label{al-2}
\end{align}
Applying \eqref{al-1} and \eqref{al-2} to \eqref{al-0}, we obtain
\begin{align*}
\Psi(H_{0,1})&=H_{0,1}+H_{n,1}+(\ve+\dfrac{\hbar}{2}n)H_{n,0}+\hbar H_{n,0}H_{0,0}\\
&\quad-\hbar\displaystyle\sum_{s \geq 0} \limits E_{n,n+1}t^{-s-1} E_{n+1,n}t^{s+1}+\hbar\displaystyle\sum_{s \geq 0} \limits E_{1,n+1}t^{-s-1} E_{n+1,1}t^{s+1}.
\end{align*}
We complete the proof.
\end{proof}
In order to prove Theorem~\ref{Main}, it is enough to show that $\Psi$ is compatible with \eqref{Eq2.1}-\eqref{Eq2.10}. By the definition of $\Psi$, $\Psi$ is compatible with \eqref{Eq2.2}, \eqref{Eq2.4} and \eqref{Eq2.10}. We will prove the compatibility with other relations of Proposition~\ref{Prop32} in the following subsections.
\subsection{Compatibility with \eqref{Eq2.3}}
The case that $i=j=0$ has already been proved in the proof of Corollary~\ref{Cor}. We only show the case that $i,j\neq 0$. Other cases are proven in a similar way. We obtain
\begin{align*}
&\quad[\Psi(X^+_{i,1}),\Psi(X^-_{j,0})]\\
&=[X^+_{i,1}-\hbar\displaystyle\sum_{s \geq 0}\limits E_{i,n+1}t^{-s-1} E_{n+1,i+1}t^{s+1},E_{j+1,j}]\\
&=\delta_{i,j}H_{i,1}-\delta_{i,j}\hbar\displaystyle\sum_{s \geq 0}\limits E_{i,n+1}t^{-s-1} E_{n+1,i}t^{s+1}+\delta_{i,j}\hbar\displaystyle\sum_{s \geq 0}\limits E_{i+1,n+1}t^{-s-1} E_{n+1,i+1}t^{s+1}\\
&=\delta_{i,j}\Psi(H_{i,1}),
\end{align*}
where the first and last equalities are due to the definition of $\Psi$ and the second equality is due to \eqref{Eq2.3}.
\subsection{Compatibility with \eqref{Eq2.5}}
We only show the case that $i=j=0$ and the sign is $+$, The other cases are proven in a similar way. By the definition of $\Psi$, we have
\begin{align}
&\quad[\Psi(\widetilde{H}_{0,1}),\Psi(X^+_{0,0})]\nonumber\\
&=[\widetilde{H}_{0,1}+\widetilde{H}_{n,1},E_{n,1}t]+(\ve+\dfrac{\hbar}{2}n)[H_{n,0},E_{n,1}t]\nonumber\\
&\quad-[\hbar\displaystyle\sum_{s \geq 0} \limits E_{n,n+1}t^{-s-1} E_{n+1,n}t^{s+1},E_{n,1}t]+[\hbar\displaystyle\sum_{s \geq 0} \limits E_{1,n+1}t^{-s-1} E_{n+1,1}t^{s+1},E_{n,1}t].\label{2.5-1}
\end{align}
By a direct computation, we obtain
\begin{align}
(\ve+\dfrac{\hbar}{2}n)[H_{n,0},E_{n,1}t]&=(\ve+\dfrac{\hbar}{2}n)E_{n,1}t\label{2.5-2}
\end{align}
and
\begin{align}
&\quad-[\hbar\displaystyle\sum_{s \geq 0} \limits E_{n,n+1}t^{-s-1} E_{n+1,n}t^{s+1},E_{n,1}t]+[\hbar\displaystyle\sum_{s \geq 0} \limits E_{1,n+1}t^{-s-1} E_{n+1,1}t^{s+1},E_{n,1}t]\nonumber\\
&=-\hbar\displaystyle\sum_{s \geq 0} \limits E_{n,n+1}t^{-s-1} E_{n+1,1}t^{s+2}-\hbar\displaystyle\sum_{s \geq 0} \limits E_{n,n+1}t^{-s} E_{n+1,1}t^{s+1}\nonumber\\
&=-2\hbar\displaystyle\sum_{s \geq 0} \limits E_{n,n+1}t^{-s} E_{n+1,1}t^{s+1}+\hbar X^+_{n,0}X^+_{0,0}.\label{2.5-2.5}
\end{align}
By \eqref{Eq2.5}-\eqref{Eq2.7}, we have
\begin{align}
&\quad[\widetilde{H}_{0,1}+\widetilde{H}_{n,1},E_{n,1}t]=[\widetilde{H}_{0,1}+\widetilde{H}_{n,1},[X^+_{n,0},X^+_{0,0}]]\nonumber\\
&=[X^+_{n,0},X^+_{0,1}-(\ve+\dfrac{n+1}{2}\hbar)X^+_{0,0}]+[X^+_{n,1}+(\ve+\dfrac{n+1}{2}\hbar)X^+_{n,0},X^+_{0,0}]\nonumber\\
&=[X^+_{n,0},X^+_{0,1}]+[X^+_{n,1},X^+_{0,0}]\nonumber\\
&=2[X^+_{n,0},X^+_{0,1}]-\dfrac{\hbar}{2}\{X^+_{0,0},X^+_{n,0}\}+(\ve+\dfrac{n+1}{2}\hbar)[X^+_{0,0},X^+_{n,0}]\nonumber\\
&=2[X^+_{n,0},X^+_{0,1}]-\dfrac{\hbar}{2}\{X^+_{0,0},X^+_{n,0}\}-(\ve+\dfrac{n+1}{2}\hbar)E_{n,1}t\nonumber\\
&=2[X^+_{n,0},X^+_{0,1}]-\hbar X^+_{n,0}X^+_{0,0}-\dfrac{\hbar}{2}[X^+_{0,0},X^+_{n,0}]-(\ve+\dfrac{n+1}{2}\hbar)E_{n,1}t,\label{2.5-3}
\end{align}
where the first equality is due to $[X^+_{n,0},X^+_{0,0}]=E_{n,1}t$, the second equality is due to \eqref{Eq2.5}-\eqref{Eq2.7}, the 4-th equality is due to \eqref{Eq2.9}.
Applying \eqref{2.5-2}-\eqref{2.5-3} to \eqref{2.5-1}, we obtain
\begin{align*}
[\Psi(\widetilde{H}_{0,1}),\Psi(X^+_{0,0})]
&=2[X^+_{n,0},X^+_{0,1}]-2\hbar\displaystyle\sum_{s \geq 0} \limits E_{n,n+1}t^{-s} E_{n+1,1}t^{s+1}\\
&=2\Psi(X^+_{0,1}).
\end{align*}
\subsection{Compatibility with \eqref{Eq2.6}}
We only prove the $+$ case. The $+$ case can be shown in the same way.
By the definition of $\Psi$, we have
\begin{align*}
&\quad[\Psi(\widetilde{H}_{0,1}),\Psi(X^+_{n-1,0})]\\
&=[\widetilde{H}_{0,1}+\widetilde{H}_{n,1},X^+_{n-1,0}]+(\ve+\dfrac{\hbar}{2}n)[H_{n,0},E_{n-1,n}]\\
&\quad-[\hbar\displaystyle\sum_{s \geq 0} \limits E_{n,n+1}t^{-s-1} E_{n+1,n}t^{s+1},E_{n-1,n}]+[\hbar\displaystyle\sum_{s \geq 0} \limits E_{1,n+1}t^{-s-1} E_{n+1,1}t^{s+1},E_{n-1,n}]\\
&=-X^+_{n-1,1}-(\ve+\dfrac{\hbar}{2}n)E_{n-1,n}+\hbar\displaystyle\sum_{s \geq 0} \limits E_{n-1,n+1}t^{-s-1} E_{n+1,n}t^{s+1}\\
&=-(\Psi(X^+_{n-1,1})+(\ve+\dfrac{\hbar}{2}n)\Psi(X^+_{n-1,0})),
\end{align*}
where the first last equalities are due to the definition of $\Psi$ and the second equality is due to \eqref{Eq2.5}.
\subsection{Compatibility with \eqref{Eq2.7}}
By the definition of $\Psi$, we have
\begin{align}
&\quad[\Psi(\widetilde{H}_{n-1,1}),\Psi(X^+_{0,0})]\nonumber\\
&=[\widetilde{H}_{n-1,1},E_{n,1}t]-[\hbar\displaystyle\sum_{s \geq 0} \limits E_{n-1,n+1}t^{-s-1} E_{n+1,n-1}t^{s+1},E_{n,1}t]\nonumber\\
&\quad+[\hbar\displaystyle\sum_{s \geq 0}\limits E_{n,n+1}t^{-s-1} E_{n+1,n}t^{s+1},E_{n,1}t].\label{2.7-1}
\end{align}
By a direct computation, we obtain
\begin{align}
&\quad-[\hbar\displaystyle\sum_{s \geq 0} \limits E_{n-1,n+1}t^{-s-1} E_{n+1,n-1}t^{s+1},E_{n,1}t]+[\hbar\displaystyle\sum_{s \geq 0}\limits E_{n,n+1}t^{-s-1} E_{n+1,n}t^{s+1},E_{n,1}t]\nonumber\\
&=0+\hbar\displaystyle\sum_{s \geq 0}\limits E_{n,n+1}t^{-s-1} E_{n+1,1}t^{s+2}\nonumber\\
&=\hbar\displaystyle\sum_{s \geq 0}\limits E_{n,n+1}t^{-s} E_{n+1,1}t^{s+1}-\hbar X^+_{n,0}X^+_{0,0}.\label{2.7-2}
\end{align}
By \eqref{Eq2.5} and \eqref{Eq2.7}, we have
\begin{align}
&\quad[\widetilde{H}_{n-1,1},E_{n,1}t]=[\widetilde{H}_{n-1,1},[X^+_{n,0},X^+_{0,0}]]=-[X^+_{n,1},X^+_{0,0}]\nonumber\\
&=-[X^+_{n,0},X^+_{0,1}]+\dfrac{\hbar}{2}\{X^+_{0,0},X^+_{n,0}\}+(\ve+\dfrac{\hbar}{2}(n+1))[X^+_{n,0},X^+_{0,0}]\nonumber\\
&=-[X^+_{n,0},X^+_{0,1}]+\hbar X^+_{n,0}X^+_{0,0}+\dfrac{\hbar}{2}[X^+_{0,0},X^+_{n,0}]+(\ve+\dfrac{\hbar}{2}(n+1))[X^+_{n,0},X^+_{0,0}]\label
{2.7-3}
\end{align}
By applying \eqref{2.7-2} and \eqref{2.7-3} to \eqref{2.7-1}, we have
\begin{align*}
&\quad[\Psi(\widetilde{H}_{n-1,1}),\Psi(X^+_{0,0})]\nonumber\\
&=-[X^+_{n,0},X^+_{0,1}]+\hbar\displaystyle\sum_{s \geq 0}\limits E_{n,n+1}t^{-s} E_{n+1,1}t^{s+1}+(\ve+\dfrac{\hbar}{2}n)[X^+_{n,0},X^+_{0,0}]\nonumber\\
&=-(\Psi(X^+_{0,1})-(\ve+\dfrac{\hbar}{2}n)\Psi(X^+_{0,0})),
\end{align*}
\subsection{Compatibility with \eqref{Eq2.8}}
We only show the case that $i=0$ and the sign is $+$. The other cases are proven in a similar way.

{\bf Case 1}: $i=0,j\neq 0,n-1$

By the definition of $\Psi$ and the assumption that $j\neq 0,n-1$, we have
\begin{align*}
&\quad[\Psi(X^+_{0,0}),\Psi(X^+_{j,1})]\\
&=[[X^+_{n,0},X^+_{0,0}],X^+_{j,1}]-[E_{n,1}t,\hbar\displaystyle\sum_{s \geq 0}\limits E_{j,n+1}t^{-s-1} E_{n+1,j+1}t^{s+1}]\\
&=[[X^+_{n,0},X^+_{0,0}],X^+_{j,1}]-\delta_{j,1}\hbar\displaystyle\sum_{s \geq 0}\limits E_{n,n+1}t^{-s} E_{n+1,j+1}t^{s+1}.
\end{align*}
and
\begin{align*}
&\quad[\Psi(X^+_{0,1}),\Psi(X^+_{j,0})]\\
&=[[X^+_{n,0},X^+_{0,1}],X^+_{j,0}]-[\hbar\displaystyle\sum_{s \geq 0}\limits E_{n,n+1}t^{-s} E_{n+1,1}t^{s+1},E_{j,j+1}]\\
&=[[X^+_{n,0},X^+_{0,1}],X^+_{j,0}]-\delta_{j,1}\hbar\displaystyle\sum_{s \geq 0}\limits E_{n,n+1}t^{-s} E_{n+1,2}t^{s+1}.
\end{align*}
Thus, by a direct computation, we have
\begin{align*}
&\quad[\Psi(X^+_{0,1}),\Psi(X^+_{j,0})]-[\Psi(X^+_{0,0}),\Psi(X^+_{j,1})]\\
&=[[X^+_{n,0},X^+_{0,1}],X^+_{j,0}]-[[X^+_{n,0},X^+_{0,0}],X^+_{j,1}].
\end{align*}
We obtain
\begin{align}
&\quad[[X^+_{n,0},X^+_{0,1}],X^+_{j,0}]-[[X^+_{n,0},X^+_{0,0}],X^+_{j,1}]\nonumber\\
&=[X^+_{n,0},([X^+_{0,1},X^+_{j,0}]-[X^+_{0,0},X^+_{j,1}])]\nonumber\\
&=\dfrac{\hbar}{2}a_{0,j}[X^+_{n,0},\{X^+_{0,0},X^+_{j,0}\}]\nonumber\\
&=\dfrac{\hbar}{2}a_{0,j}\{[X^+_{n,0},X^+_{0,0}],X^+_{j,0}\}\nonumber\\
&=\dfrac{\hbar}{2}a_{0,j}\{\Psi(X^+_{0,0}),\Psi(X^+_{j,0})\},\label{aaa}
\end{align}
where the first equality is due to \eqref{gather1} and the assumption that $j\neq 0,n-1$, the second equality is due to \eqref{Eq2.8} and the last equlity is due to the assumption that $j\neq 0,n-1$.

{\bf Case 2}: $i=j=0$

In this case, \eqref{Eq2.8} is equivalent to
\begin{equation}
[X^+_{0,1},X^+_{0,0}]=\hbar(X^+_{0,0})^2.\label{Eq2.11}
\end{equation}
We will prove the compatibility with \eqref{Eq2.11}. By the definition of $\Psi$, we have
\begin{align}
&\quad[\Psi(X^+_{0,1}),\Psi(X^+_{0,0})]\nonumber\\
&=[[X^+_{n,0},X^+_{0,1}],[X^+_{n,0},X^+_{0,0}]]-[\hbar\displaystyle\sum_{s \geq 0} \limits E_{n,n+1}t^{-s} E_{n+1,1}t^{s+1},E_{n,1}t]\nonumber\\
&=[[X^+_{n,0},X^+_{0,1}],[X^+_{n,0},X^+_{0,0}]]-0.\label{2.8-1}
\end{align}
By \eqref{Eq2.10}, we obtain
\begin{align}
&\quad[[X^+_{n,0},X^+_{0,1}],[X^+_{n,0},X^+_{0,0}]]\nonumber\\
&=[X^+_{n,0},[X^+_{0,1},[X^+_{n,0},X^+_{0,0}]]]+[[X^+_{n,0},[X^+_{n,0},X^+_{0,0}]],X^+_{0,1}]\nonumber\\
&=[X^+_{n,0},[X^+_{0,1},[X^+_{n,0},X^+_{0,0}]]]+0,\label{2.8-1.5}
\end{align}
where the last equality is due to \eqref{Eq2.10}.
We obtain
\begin{align}
&\quad[X^+_{0,1},[X^+_{n,0},X^+_{0,0}]]\nonumber\\
&=\dfrac{1}{2}[X^+_{0,1},[X^+_{n,0},X^+_{0,0}]]+\dfrac{1}{2}([[X^+_{0,1},X^+_{n,0}],X^+_{0,0}]+[X^+_{n,0},[X^+_{0,1},X^+_{0,0}]])\nonumber\\
&=\dfrac{1}{2}([X^+_{0,1},[X^+_{n,0},X^+_{0,0}]]+[[X^+_{0,1},X^+_{n,0}],X^+_{0,0}])+\dfrac{1}{2}[X^+_{n,0},[X^+_{0,1},X^+_{0,0}]]\nonumber\\
&=0+\dfrac{\hbar}{2}[X^+_{n,0},(X^+_{0,0})^2]\nonumber\\
&=\dfrac{\hbar}{2}\{[X^+_{n,0},X^+_{0,0}],X^+_{0,0}\},\label{2.8-2}
\end{align}
where the third equality is due to \eqref{gather2} and \eqref{Eq2.11}. By applying \eqref{2.8-1.5} and \eqref{2.8-2} to \eqref{2.8-1}, we obtain
\begin{align*}
&\quad[\Psi(X^+_{0,1}),\Psi(X^+_{0,0})]\nonumber\\
&=[X^+_{n,0},\dfrac{\hbar}{2}\{[X^+_{n,0},X^+_{0,0}],X^+_{0,0}\}]\\
&=\dfrac{\hbar}{2}\{[X^+_{n,0},X^+_{0,0}],[X^+_{n,0},X^+_{0,0}]\}=\hbar(\Psi(X^+_{0,0}))^2
\end{align*}
by \eqref{Eq2.10}.
\subsection{Compatibility with \eqref{Eq2.9}}
We only show the $+$ case. The $-$ case can be proven in a similar way. By the definition of $\Psi$, we have
\begin{align}
&\quad[\Psi(X^+_{0,0}),\Psi(X^+_{n-1,1})]\nonumber\\
&=[[X^+_{n,0},X^+_{0,0}],X^+_{n-1,1}]-[E_{n,1}t,\hbar\displaystyle\sum_{s \geq 0}\limits E_{n-1,n+1}t^{-s-1} E_{n+1,n}t^{s+1}]\nonumber\\
&=[[X^+_{n,0},X^+_{0,0}],X^+_{n-1,1}]+\hbar\displaystyle\sum_{s \geq 0}\limits E_{n-1,n+1}t^{-s-1} E_{n+1,1}t^{s+2}\label{2.9.10}
\end{align}
and
\begin{align}
&\quad[\Psi(X^+_{0,1}),\Psi(X^+_{n-1,0})]\\
&=[[X^+_{n,0},X^+_{0,1}],X^+_{n-1,0}]-[\hbar\displaystyle\sum_{s \geq 0}\limits E_{n,n+1}t^{-s} E_{n+1,1}t^{s+1},E_{n-1,n}]\\
&=[[X^+_{n,0},X^+_{0,1}],X^+_{n-1,0}]+\hbar\displaystyle\sum_{s \geq 0}\limits E_{n-1,n+1}t^{-s} E_{n+1,1}t^{s+1}.\label{2.9.11}
\end{align}
By comparing the right hand sides of \eqref{2.9.10} and \eqref{2.9.11}, we have
\begin{align}
&\quad[\Psi(X^+_{0,1}),\Psi(X^+_{n-1,0})]-[\Psi(X^+_{0,0}),\Psi(X^+_{n-1,1})]\nonumber\\
&=[[X^+_{n,0},X^+_{0,1}],X^+_{n-1,0}]+\hbar E_{n-1,n+1}E_{n+1,1}t-[[X^+_{n,0},X^+_{0,0}],X^+_{n-1,1}].\label{Eq2.9-1}
\end{align}
We obtain
\begin{align}
&\quad[[X^+_{n,0},X^+_{0,1}],X^+_{n-1,0}]\nonumber\\
&=[[X^+_{n,1},X^+_{0,0}]+\dfrac{\hbar}{2}\{X^+_{0,0},X^+_{n,0}\}-(\ve+\dfrac{\hbar}{2}(n+1))[X^+_{0,0},X^+_{n-1,0}],X^+_{n-1,0}]\nonumber\\
&=[[X^+_{n,1},X^+_{0,0}],X^+_{n-1,0}]+\dfrac{\hbar}{2}\{X^+_{0,0},[X^+_{n,0},X^+_{n-1,0}]\}-(\ve+\dfrac{\hbar}{2}(n+1))[[X^+_{0,0},X^+_{n,0}],X^+_{n-1,0}]\nonumber\\
&=[[X^+_{n,1},X^+_{0,0}],X^+_{n-1,0}]-\dfrac{\hbar}{2}\{E_{n+1,1}t,E_{n-1,n+1}\}-(\ve+\dfrac{\hbar}{2}(n+1))E_{n-1,1}t,\label{Eq2.9-2}
\end{align}
where the second equality is due to \eqref{Eq2.9}.
By the similar way to \eqref{aaa}, we have
\begin{align}
&\quad[[X^+_{n,1},X^+_{0,0}],X^+_{n-1,0}]-[[X^+_{n,0},X^+_{0,0}],X^+_{n-1,1}]\nonumber\\
&=-\dfrac{\hbar}{2}\{[X^+_{n,0},X^+_{0,0}],X^+_{n-1,0}\}.\label{2.9-2}
\end{align}
By applying \eqref{2.9-2} to \eqref{Eq2.9-1}, we have
\begin{align*}
&\quad[\Psi(X^+_{0,1}),\Psi(X^+_{n-1,0})]-[\Psi(X^+_{0,0}),\Psi(X^+_{n-1,1})]\nonumber\\
&=-\dfrac{\hbar}{2}\{[X^+_{n,0},X^+_{0,0}],X^+_{n-1,0}\}-\dfrac{\hbar}{2}\{E_{n+1,1}t,E_{n-1,n+1}\}\nonumber\\
&\quad-(\ve+\dfrac{\hbar}{2}(n+1))E_{n-1,1}t+\hbar E_{n-1,n+1}E_{n+1,1}t\\
&=-\dfrac{\hbar}{2}\{[X^+_{n,0},X^+_{0,0}],X^+_{n-1,0}\}-(\ve+\dfrac{\hbar}{2}n)E_{n-1,1}t\\
&=-\dfrac{\hbar}{2}\{\Psi(X^+_{0,0}),\Psi(X^+_{n-1,0})\}+(\ve+\dfrac{\hbar}{2}n)[\Psi(X^+_{0,0}),\Psi(X^+_{n-1,0})].
\end{align*}
\subsection{Compatibility with \eqref{Eq2.1}}
By the definition of $\Psi(H_{i,1})$, $\Psi$ is compatible with \eqref{Eq2.1} in the case that $r+s\leq 1$. Thus, it is enough to prove $[\Psi(\widetilde{H}_{i,1}),\Psi(\widetilde{H}_{j,1})]=0$. We only show the case that $i,j\neq0$. The case that $i=0$ or $j=0$ can be proven in a similar way. By the definition of $\Psi$, we have
\begin{align*}
[\Psi(\widetilde{H}_{i,1}),\Psi(\widetilde{H}_{j,1})]
&=[\widetilde{H}_{i,1},\widetilde{H}_{j,1}]-[\widetilde{H}_{i,1},P_j-P_{j+1}]+[\widetilde{H}_{j,1},P_i-P_{i+1}]+[P_i-P_{i+1},P_j-P_{j+1}],
\end{align*}
where $P_i=\hbar\displaystyle\sum_{s \geq 0} \limits E_{i,n+1}t^{-s-1} E_{n+1,i}t^{s+1}$.
By the definition of $J(h_i)$, we have
\begin{align}
&\quad-[\widetilde{H}_{i,1},P_j-P_{j+1}]+[\widetilde{H}_{j,1},P_i-P_{i+1}]\nonumber\\
&=-[J(h_i),P_j-P_{j+1}]+[J(h_j),P_i-P_{i+1}]+[A_i-A_{i+1},P_j-P_{j+1}]-[A_j-A_{j+1},P_j-P_{j+1}].\label{551}
\end{align}
By Lemma~\ref{J} and the definition of $P_i$, we find that the sum of the first two terms of the right hand side of \eqref{551} are equal to zero. Thus, it is enough to show that
\begin{align}
[A_i,P_j]-[A_j,P_i]+[P_i,P_j]=0.\label{concl}
\end{align}
By a direct computation, we obtain
\begin{align}
[P_i,P_j]
&=\hbar^2\displaystyle\sum_{s,v \geq 0}\limits E_{j,n+1}t^{-v-1}E_{i,j}t^{v-s}E_{n+1,i}t^{s+1}\nonumber\\
&\quad-\hbar^2\displaystyle\sum_{s,v \geq 0}\limits E_{i,n+1}t^{-s-1} E_{j,i}t^{s-v}E_{n+1,j}t^{v+1}.\label{551-0}
\end{align}
By the definition of $A_i$, we can divide $[A_i,P_j]$ into four pieces:
\begin{align}
[A_i,P_j]
&=[\dfrac{\hbar}{2}\sum_{\substack{s\geq0\\u>i}}\limits E_{u,i}t^{-s}E_{i,u}t^s,P_j]-[\dfrac{\hbar}{2}\sum_{\substack{s\geq0\\i>u}}\limits E_{i,u}t^{-s}E_{u,i}t^s,P_j]\nonumber\\
&\quad+[\dfrac{\hbar}{2}\sum_{\substack{s\geq0\\u<i}}\limits E_{u,i}t^{-s-1}E_{i,u}t^{s+1},P_j]-[\dfrac{\hbar}{2}\sum_{\substack{s\geq0\\i<u}}\limits E_{i,u}t^{-s-1}E_{u,i}t^{s+1},P_j].\label{9112}
\end{align}
We compute the right hand  side of \eqref{9112}. By a direct computation, we obtain
\begin{align}
&\quad[\dfrac{\hbar}{2}\sum_{\substack{s\geq0\\u>i}}\limits E_{u,i}t^{-s}E_{i,u}t^s,P_j]\nonumber\\
&=\delta(j>i)\dfrac{\hbar^2}{2}\sum_{\substack{s,v\geq0}}\limits E_{j,i}t^{-s}E_{i,n+1}t^{s-v-1}E_{n+1,j}t^{v+1}\nonumber\\
&\quad+\dfrac{\hbar^2}{2}\sum_{\substack{s,v\geq0}}\limits E_{n+1,i}t^{-s}E_{j,n+1}t^{-v-1}E_{i,j}t^{s+v+1}-\dfrac{\hbar^2}{2}\sum_{\substack{s,v\geq0\\u>i}}\limits\delta_{i,j}E_{u,i}t^{-s}E_{j,n+1}t^{-v-1}E_{n+1,u}t^{s+v+1}\nonumber\\
&\quad+\dfrac{\hbar^2}{2}\sum_{\substack{s,v\geq0}}\limits\delta_{i,j}E_{u,n+1}t^{-s-v-1}E_{n+1,j}t^{v+1}E_{i,u}t^s-\dfrac{\hbar^2}{2}\sum_{\substack{s,v\geq0}}\limits E_{j,i}t^{-s-v-1}E_{n+1,j}t^{v+1}E_{i,n+1}t^s\nonumber\\
&\quad-\delta(j>i)\dfrac{\hbar^2}{2}\sum_{\substack{s,v\geq0}}\limits E_{j,n+1}t^{-v-1}E_{n+1,i}t^{v+1-s}E_{i,j}t^s,\label{551-1}\\
&\quad-[\dfrac{\hbar}{2}\sum_{\substack{s\geq0\\i>u}}\limits E_{i,u}t^{-s}E_{u,i}t^s,P_j]\nonumber\\
&=-\dfrac{\hbar^2}{2}\sum_{\substack{s,v\geq0}}\limits \delta_{i,j}E_{i,u}t^{-s}E_{u,n+1}t^{s-v-1}E_{n+1,j}t^{v+1}\nonumber\\
&\quad+\delta(i>j)\dfrac{\hbar^2}{2}\sum_{\substack{s,v\geq0}}\limits E_{i,j}t^{-s}E_{j,n+1}t^{-v-1}E_{n+1,i}t^{s+v+1}\nonumber\\
&\quad-\delta(i>j)\dfrac{\hbar^2}{2}\sum_{\substack{s,v\geq0}}\limits E_{i,n+1}t^{-s-v-1}E_{n+1,j}t^{v+1}E_{j,i}t^s\nonumber\\
&\quad+\dfrac{\hbar^2}{2}\sum_{\substack{s,v\geq0}}\limits\delta_{i,j}E_{j,n+1}t^{-v-1}E_{n+1,u}t^{v+1-s}E_{u,i}t^s,\label{551-2}\\
&\quad[\dfrac{\hbar}{2}\sum_{\substack{s\geq0\\u<i}}\limits E_{u,i}t^{-s-1}E_{i,u}t^{s+1},P_j]\nonumber\\
&=\delta(j<i)\dfrac{\hbar^2}{2}\sum_{\substack{s,v\geq0}}\limits E_{j,i}t^{-s-1}E_{i,n+1}t^{s-v}E_{n+1,j}t^{v+1}\nonumber\\
&\quad+\dfrac{\hbar^2}{2}\sum_{\substack{s,v\geq0}}\limits \delta_{i,j}E_{u,i}t^{-s-1}E_{j,n+1}t^{-v-1}E_{n+1,u}t^{s+v+1}\nonumber\\
&\quad+\dfrac{\hbar^2}{2}\sum_{\substack{s,v\geq0\\u<i}}\limits \delta_{i,j}E_{u,n+1}t^{-s-v-2}E_{n+1,j}t^{v+1}E_{i,u}t^{s+1}\nonumber\\
&\quad-\delta(j<i)\dfrac{\hbar^2}{2}\sum_{\substack{s,v\geq0}}\limits E_{j,n+1}t^{-v-1}E_{n+1,i}t^{v-s}E_{i,j}t^{s+1},\label{551-3}\\
&\quad-[\dfrac{\hbar}{2}\sum_{\substack{s\geq0\\i<u}}\limits E_{i,u}t^{-s-1}E_{u,i}t^{s+1},P_j]\nonumber\\
&=-\dfrac{\hbar^2}{2}\sum_{\substack{s,v\geq0}}\limits\delta_{i,j}E_{i,u}t^{-s-1}E_{u,n+1}t^{s-v}E_{n+1,j}t^{v+1}+\dfrac{\hbar^2}{2}\sum_{\substack{s,v\geq0}}\limits E_{i,n+1}t^{-s}E_{j,i}t^{s-v-1}E_{n+1,j}t^{v+1}\nonumber\\
&\quad+\delta(i<j)\dfrac{\hbar^2}{2}\sum_{\substack{s,v\geq0}}\limits E_{i,j}t^{-s-1}E_{j,n+1}t^{-v-1}E_{n+1,i}t^{s+v+2}\nonumber\\
&\quad-\delta(i<j)\dfrac{\hbar^2}{2}\sum_{\substack{s\geq0}}\limits E_{i,n+1}t^{-s-v-2}E_{n+1,j}t^{v+1}E_{j,i}t^{s+1}\nonumber\\
&\quad-\dfrac{\hbar^2}{2}\sum_{\substack{s,v\geq0}}\limits E_{j,n+1}t^{-v-1}E_{i,j}t^{v-s}E_{n+1,i}t^{s+1}+\dfrac{\hbar^2}{2}\sum_{\substack{s,v\geq0}}\limits\delta_{i,j}E_{j,n+1}t^{-v-1}E_{n+1,u}t^{v-s}E_{u,i}t^{s+1}.\label{551-4}
\end{align}
Here after, we denote $(\text{equation number})_{a,b}$ means that the value of $(\text{equation number})$ at $i=a,j=b$. Moreover, we denote the $r$-th term of the right hand side of $(\text{equation number})$ by $(\text{equation number})_r$.

By the definition of $A_i$, we have
\begin{align}
[A_i,P_j]-[A_j,P_i]&=\eqref{551-1}_{i,j}+\eqref{551-2}_{i,j}+\eqref{551-3}_{i,j}+\eqref{551-4}_{i,j}\nonumber\\
&\quad-\eqref{551-1}_{j,i}-\eqref{551-2}_{j,i}-\eqref{551-3}_{j,i}-\eqref{551-4}_{j,i}.\label{9113}
\end{align}
By the definition, we find that the terms containing $\delta_{i,j}$ in the right hand side of \eqref{9113} vanish. By a direct computation, we can compute the sum of the terms containing $\delta(j>i)$ in the right hand side of \eqref{9113}:
\begin{align}
&\quad\eqref{551-1}_{i,j,1}+\eqref{551-1}_{i,j,6}-\eqref{551-2}_{j,i,2}-\eqref{551-2}_{j,i,,3}\nonumber\\
&\qquad\qquad-\eqref{551-3}_{j,i,1}-\eqref{551-3}_{j,i,4}+\eqref{551-4}_{i,j,3}+\eqref{551-4}_{i,j,4}\nonumber\\
&=\delta(j>i)\dfrac{\hbar^2}{2}\sum_{\substack{s,v\geq0}}\limits E_{j,i}t^{-s-v-1}E_{i,n+1}t^{s}E_{n+1,j}t^{v+1}\nonumber\\
&\quad-\delta(j>i)\dfrac{\hbar^2}{2}\sum_{\substack{s,v\geq0}}\limits E_{j,n+1}t^{-v-1}E_{n+1,i}t^{-s}E_{i,j}t^{s+v+1}\nonumber\\
&\quad-\delta(i<j)\dfrac{\hbar^2}{2}\sum_{\substack{s,v\geq0}}\limits E_{i,j}t^{-s-v-1}E_{j,n+1}t^{s}E_{n+1,i}t^{v+1}\nonumber\\
&\quad+\delta(i<j)\dfrac{\hbar^2}{2}\sum_{\substack{s,v\geq0}}\limits E_{i,n+1}t^{-v-1}E_{n+1,j}t^{-s}E_{j,i}t^{s+v+1}.\label{551-5}
\end{align}
Similarly, we can compute the sum of the terms containing $\delta(j<i)$ in the right hand side of \eqref{9113}:
\begin{align}
&-\eqref{551-1}_{j,i,1}-\eqref{551-1}_{j,i,6}+\eqref{551-2}_{i,j,2}+\eqref{551-2}_{i,j,3}\nonumber\\
&\qquad\qquad+\eqref{551-3}_{i,j,1}+\eqref{551-3}_{i,j,4}-\eqref{551-4}_{j,i3}-\eqref{551-4}_{j,i,4}\nonumber\\
&=-\delta(i>j)\dfrac{\hbar^2}{2}\sum_{\substack{s,v\geq0}}\limits E_{i,j}t^{-s-v-1}E_{j,n+1}t^{s}E_{n+1,i}t^{v+1}\nonumber\\
&\quad+\delta(i<j)\dfrac{\hbar^2}{2}\sum_{\substack{s,v\geq0}}\limits E_{i,n+1}t^{-v-1}E_{n+1,j}t^{-s}E_{j,i}t^{s+v+1}\nonumber\\
&\quad+\delta(j<i)\dfrac{\hbar^2}{2}\sum_{\substack{s,v\geq0}}\limits E_{j,i}t^{-s-v-1}E_{i,n+1}t^{s}E_{n+1,j}t^{v+1}\nonumber\\
&\quad-\delta(j<i)\dfrac{\hbar^2}{2}\sum_{\substack{s,v\geq0}}\limits E_{j,n+1}t^{-v-1}E_{n+1,i}t^{-s}E_{i,j}t^{s+v+1}.\label{551-7}
\end{align}
By a direct computation, we obtain
\begin{align*}
&\quad\eqref{551-5}_2+\eqref{551-7}_4+\eqref{551-1}_{i,j,2}\\
&=\delta(i\neq j)\dfrac{\hbar^2}{2}\sum_{\substack{s,v\geq0}}\limits [E_{n+1,i}t^{-s},E_{j,n+1}t^{-v-1}]E_{i,j}t^{s+v+1}=0.
\end{align*}
Similarly, we have
\begin{align*}
\eqref{551-5}_3+\eqref{551-7}_1-\eqref{551-1}_{j,i,5}&=0.
\end{align*}
Then, we find that $[A_i,P_j]-[A_j,P_i]+[P_i,P_j]$ is equal to the sum of the following four terms:
\begin{gather*}
\eqref{551-5}_1+\eqref{551-7}_3+\eqref{551-1}_{i,j,5},\\
\eqref{551-5}_4+\eqref{551-7}_2-\eqref{551-1}_{j,i,2},\\
\eqref{551-0}_1-\eqref{551-4}_{j,i,2}+\eqref{551-4}_{i,j,5},\\
\eqref{551-0}_2+\eqref{551-4}_{i,j,2}-\eqref{551-4}_{j,i,5}.
\end{gather*}
By a direct computation, these four sums are equal to zero.
We complete the proof of the compatibility with \eqref{Eq2.1}.
\section{The rectangular $W$-algebra $\mathcal{W}^k(\mathfrak{gl}(2n),(2^n))$}
Let us set some notations of a vertex algebra. For a vertex algebra $V$, we denote the generating field associated with $v\in V$ by $v(z)=\displaystyle\sum_{n\in\mathbb{Z}}\limits v_{(n)}z^{-n-1}$. We also denote the OPE of $V$ by
\begin{equation*}
u(z)v(w)\sim\displaystyle\sum_{s\geq0}\limits \dfrac{(u_{(s)}v)(w)}{(z-w)^{s+1}}
\end{equation*}
for all $u, v\in V$. We denote the vacuum vector (resp.\ the translation operator) by $|0\rangle$ (resp.\ $\partial$).

We denote the universal affine vertex algebra associated with a finite dimensional Lie algebra $\mathfrak{g}$ and its inner product $\kappa$ by $V^\kappa(\mathfrak{g})$. By the PBW theorem, we can identify $V^\kappa(\mathfrak{g})$ with $U(t^{-1}\mathfrak{g}[t^{-1}])$. In order to simplify the notation, here after, we denote the generating field $(ut^{-1})(z)$ as $u(z)$. By the definition of $V^\kappa(\mathfrak{g})$, the generating fields $u(z)$ and $v(z)$ satisfy the OPE
\begin{gather}
u(z)v(w)\sim\dfrac{[u,v](w)}{z-w}+\dfrac{\kappa(u,v)}{(z-w)^2}\label{OPE1}
\end{gather}
for all $u,v\in\mathfrak{g}$. For a matrix unit $e_{i,j}$, we denote $e_{i,j}t^{-m}\in U(t^{-1}\mathfrak{g}[t^{-1}])=V^\kappa(\mathfrak{g})$ by $e_{i,j}[-m]$.

The $W$-algebra $\mathcal{W}^k(\mathfrak{g},f)$ is a vertex algebra associated with the reductive Lie algebra $\mathfrak{g}$ and a nilpotent element $f$. We call the $W$-algebra associated with $\mathfrak{gl}(ln)$ and a nilpotent element of type $(l^n)$ the rectangular $W$-algebra and denote it by $\mathcal{W}^k(\mathfrak{gl}(ln),(l^n))$.
In this article, we only consider the case that $l=2$. The nilpotent element is
\begin{equation*}
f=\sum_{u=1}^n\limits e_{n+u,u}.
\end{equation*}
By Theorem 3.1 and Corollary 3.2 in \cite{AM}, we obtain the following theorem.
\begin{Theorem}[Theorem 3.1 and Corollary 3.2 in \cite{AM}]
\begin{enumerate}
\item We define the inner product on $\mathfrak{gl}(n)$ by
\begin{equation*}
\kappa(e_{i,j},e_{p,q})=\delta_{j,p}\delta_{i,q}\alpha+\delta_{i,j}\delta_{p,q},
\end{equation*}
where $\alpha=k+n$.
Then, the rectangular $W$-algebra $\mathcal{W}^k(\mathfrak{gl}(2n),(2^n))$ can be realized as a vertex subalgebra of $V^{\kappa}(\mathfrak{gl}(n))^{\otimes 2}$.
\item The $W$-algebra $\mathcal{W}^k(\mathfrak{g},f)$ has the following strong generators:
\begin{align*}
W^{(1)}_{i,j}&=e^{(1)}_{i,j}[-1]+e^{(2)}_{i,j}[-1],\\
W^{(2)}_{i,j}&=\sum_{1\leq u\leq n}\limits e^{(1)}_{u,j}[-1]e^{(2)}_{i,u}[-1]-\alpha e^{(1)}_{i,j}[-1]
\end{align*} 
for $1\leq i,j\leq n$, where $e^{(1)}_{i,j}[-1]=e_{i,j}[-1]\otimes 1\in V^{\kappa}(\mathfrak{gl}(n))^{\otimes 2}$ and $e^{(2)}_{i,j}[-1]=1\otimes e_{i,j}[-1]\in V^{\kappa}(\mathfrak{gl}(n))^{\otimes 2}$.
\end{enumerate}
\end{Theorem}
\begin{Remark}
We note that $W^{(2)}_{i,j}$ in this article is different from the one in \cite{AM}. We shift $W^{(2)}_{i,j}$ in this article is corresponding to $W^{(2)}_{j,i}-\alpha\partial W^{(1)}_{j,i}$ in \cite{AM}.
\end{Remark}
We can compute all OPEs of these strong generators. The computation can be done by using the computation process in the appendix of \cite{U5}.
\begin{Theorem}\label{Tho1}
\begin{enumerate}
\item The following equations hold:
\begin{gather*}
(W^{(1)}_{p,q})_{(0)}W^{(1)}_{i,j}=\delta_{q,i}W^{(1)}_{p,j}-\delta_{p,j}W^{(1)}_{i,q},\\
(W^{(1)}_{p,q})_{(1)}W^{(1)}_{i,j}=2\delta_{q,i}\delta_{p,j}\alpha|0\rangle+\delta_{p,q}\delta_{i,j}(1+1)|0\rangle,\\
(W^{(1)}_{p,q})_{(s)}W^{(1)}_{i,j}=0\text{ for all }s>1.
\end{gather*}
\item The following four equations hold:
\begin{align*}
(W^{(1)}_{p,q})_{(0)}W^{(2)}_{i,j}
&=-\delta_{p,j}W^{(2)}_{i,q}+\delta_{i,q}W^{(2)}_{p,j},\\
(W^{(1)}_{p,q})_{(1)}W^{(2)}_{i,j}&=\delta_{p,j}\alpha W^{(1)}_{i,q}+\delta_{p,q}W^{(1)}_{i,j},\\
(W^{(1)}_{p,q})_{(2)}W^{(2)}_{i,j}
&=-2\delta_{q,i}\delta_{p,j}\alpha^2|0\rangle-2\delta_{p,q}\delta_{i,j}\alpha|0\rangle,\\
(W^{(1)}_{p,q})_{(s)}W^{(2)}_{i,j}&=0\text{ for all }s>2.
\end{align*}
\item The following relations hold:
\begin{align}
&\quad (W^{(2)}_{p,q})_{(0)}W^{(2)}_{i,j}\nonumber\\
&=(W^{(2)}_{p,j})_{(-1)}W^{(1)}_{i,q}-(W^{(1)}_{p,j})_{(-1)}W^{(2)}_{i,q}+\alpha(\partial W^{(1)}_{p,j})_{(-1)}W^{(1)}_{i,q}+(\partial W^{(1)}_{p,q})_{(-1)}W^{(1)}_{i,j}\nonumber\\
&\quad-\delta_{q,i}\alpha\partial W^{(2)}_{p,j}-\delta_{i,q}\dfrac{2\alpha^2+1}{2}\partial^2W^{(1)}_{p,j}-\delta_{i,j}\partial W^{(2)}_{p,q}-\dfrac{3}{2}\delta_{i,j}\alpha\partial^2W^{(1)}_{p,q},\label{OPE3-1}\\
&\quad (W^{(2)}_{p,q})_{(1)}W^{(2)}_{i,j}\nonumber\\
&=\alpha(W^{(1)}_{p,j})_{(-1)}W^{(1)}_{i,q}+(W^{(1)}_{p,q})_{(-1)}W^{(1)}_{i,j}\nonumber\\
&\quad-\delta_{q,i}\alpha W^{(2)}_{p,j}-2\delta_{q,i}\alpha^2\partial W^{(1)}_{p,j}-\delta_{p,j}\alpha W^{(2)}_{i,q}-\delta_{i,j}(1)W^{(2)}_{p,q}-2\delta_{i,j}\alpha\partial W^{(1)}_{p,q}-\delta_{p,q}W^{(2)}_{i,j},\label{OPE3-2}\\
&\quad(W^{(2)}_{p,q})_{(2)}W^{(2)}_{i,j}\nonumber\\
&=\delta_{p,j}\alpha(2\alpha-1)W^{(1)}_{i,q}-\delta_{i,j}\alpha W^{(1)}_{p,q}-\delta_{i,q}\alpha(2\alpha-1)W^{(1)}_{i,q}+\delta_{p,q}\alpha W^{(1)}_{i,j},\label{OPE3-3}\\
&\quad(W^{(2)}_{p,q})_{(3)}W^{(2)}_{i,j}\nonumber\\
&=(1+\alpha^2-6\alpha^2)\delta_{p,q}\delta_{i,j}|0\rangle+(2\alpha-6\alpha^3)\delta_{p,j}\delta_{i,q}|0\rangle,\label{OPE3-4}\\
&\quad(W^{(2)}_{p,q})_{(s)}W^{(2)}_{i,j}=0\text{ for all }s>0.
\end{align}
\end{enumerate}
\end{Theorem}
We note that these OPEs are only dependent on $\alpha$ and independent of $n$.
Thus, by Theorem~\ref{Tho1}, we find the following embedding:
\begin{equation*}
\mathcal{W}^{k+1}(\mathfrak{gl}(2n),(2^n))\to \mathcal{W}^{k}(\mathfrak{gl}(2(n+1)),(2^{n+1})),\ W^{(u)}_{i,j}\mapsto W^{(u)}_{i,j}.
\end{equation*}
\section{The relationship between homomorphism $\Psi$ and the rectangular $W$-algebra $\mathcal{W}^k(\mathfrak{gl}(2n),(2^n))$}
Let us recall the definition of a universal enveloping algebra of a vertex algebra in the sense of \cite{FZ} and \cite{MNT}.
For any vertex algebra $V$, let $L(V)$ be the Borchards Lie algebra, that is,
\begin{align}
 L(V)=V{\otimes}\mathbb{C}[t,t^{-1}]/\text{Im}(\partial\otimes\id +\id\otimes\frac{d}{d t})\label{844},
\end{align}
where the commutation relation is given by
\begin{align*}
 [ut^a,vt^b]=\sum_{r\geq 0}\begin{pmatrix} a\\r\end{pmatrix}(u_{(r)}v)t^{a+b-r}
\end{align*}
for all $u,v\in V$ and $a,b\in \mathbb{Z}$. Now, we define the universal enveloping algebra of $V$.
\begin{Definition}[Section~6 in \cite{MNT}]\label{Defi}
We set $\mathcal{U}(V)$ as the quotient algebra of the standard degreewise completion of the universal enveloping algebra of $L(V)$ by the completion of the two-sided ideal generated by
\begin{gather}
(u_{(a)}v)t^b-\sum_{i\geq 0}
\begin{pmatrix}
 a\\i
\end{pmatrix}
(-1)^i(ut^{a-i}vt^{b+i}-(-1)^avt^{a+b-i}ut^{i}),\label{241}\\
|0\rangle t^{-1}-1.\label{242}
\end{gather}
We call $\mathcal{U}(V)$ the universal enveloping algebra of $V$.
\end{Definition}
In \cite{U4} Theorem~5.1, the author constructed a surjective homomorphism from the affine super Yangian to the universal enveloping algebra of a rectangular $W$-superalgebra. Setting $m=n$, $n=0$ and $l=2$, we obtain the following theorem.
\begin{Theorem}\label{Maim}
Suppose that $\hbar=-1$ and $\ve=-\alpha$.
There exists an algebra homomorphism 
\begin{equation*}
\Phi^{n}\colon Y_{\hbar,\ve}(\widehat{\mathfrak{sl}}(n))\to \mathcal{U}(\mathcal{W}^{k}(\mathfrak{gl}(2n),(2^{n})))
\end{equation*} 
determined by
\begin{gather*}
\Phi^n(H_{i,0})=\begin{cases}
W^{(1)}_{n,n}-W^{(1)}_{1,1}+2\alpha&\text{ if }i=0,\\
W^{(1)}_{i,i}-W^{(1)}_{i+1,i+1}&\text{ if }i\neq 0,
\end{cases}\\
\Phi^n(X^+_{i,0})=\begin{cases}
W^{(1)}_{n,1}t&\text{ if }i=0,\\
W^{(1)}_{i,i+1}&\text{ if }i\neq 0,
\end{cases}
\quad \Phi^n(X^-_{i,0})=\begin{cases}
W^{(1)}_{1,n}t^{-1}&\text{ if }i=0,\\
W^{(1)}_{i+1,i}&\text{ if }i\neq 0,
\end{cases}
\end{gather*}
\begin{align*}
\Phi^n(H_{i,1})&=\begin{cases}
W^{(2)}_{n,n}t-W^{(2)}_{1,1}t+\alpha W^{(1)}_{n,n}- 2\alpha\Phi^n(H_{0,0}) +W^{(1)}_{n,n} (W^{(1)}_{1,1}-2\alpha)\\
\quad-\displaystyle\sum_{s \geq 0} \limits\displaystyle\sum_{u=1}^{n}\limits W^{(1)}_{n,u}t^{-s} W^{(1)}_{u,n}t^s+\displaystyle\sum_{s \geq 0}\displaystyle\sum_{u=1}^{n}\limits W^{(1)}_{1,u}t^{-s-1} W^{(1)}_{u,1}t^{s+1},\\
\qquad\qquad\qquad\qquad\qquad\qquad\qquad\qquad\qquad\qquad\qquad\qquad\qquad\qquad \text{ if }i=0,\\
 W^{(2)}_{i,i}t- W^{(2)}_{i+1,i+1}t+\dfrac{i}{2}\Phi^n(H_{i,0})+W^{(1)}_{i,i}W^{(1)}_{i+1,i+1}\\
\quad-\displaystyle\sum_{s \geq 0}  \limits\displaystyle\sum_{u=1}^{i}\limits W^{(1)}_{i,u}t^{-s}W^{(1)}_{u,i}t^s-\displaystyle\sum_{s \geq 0} \limits\displaystyle\sum_{u=i+1}^{n}\limits  W^{(1)}_{i,u}t^{-s-1} W^{(1)}_{u,i}t^{s+1}\\
\quad+\displaystyle\sum_{s \geq 0}\limits\displaystyle\sum_{u=1}^{i}\limits W^{(1)}_{i+1,u}t^{-s} W^{(1)}_{u,i+1}t^s+\displaystyle\sum_{s \geq 0}\limits\displaystyle\sum_{u=i+1}^{n} \limits W^{(1)}_{i+1,u}t^{-s-1} W^{(1)}_{u,i+1}t^{s+1}\\
\qquad\qquad\qquad\qquad\qquad\qquad\qquad\qquad\qquad\qquad\qquad\qquad\qquad\qquad i\neq0,
\end{cases}
\end{align*}
\begin{align*}
\Phi^n(X^+_{i,1})&=\begin{cases}
W^{(2)}_{n,1}t^2+\alpha W^{(1)}_{n,1}t-
2\alpha \Phi^n(X_{0,0}^{+})-\displaystyle\sum_{s \geq 0} \limits\displaystyle\sum_{u=1}^{n}\limits W^{(1)}_{n,u}t^{-s} W^{(1)}_{u,1}t^{s+1}\\
\qquad\qquad\qquad\qquad\qquad\qquad\qquad\qquad\qquad\qquad\qquad\qquad\qquad\qquad \text{ if $i = 0$},\\
W^{(2)}_{i,i+1}t+\dfrac{i}{2}\Phi^n(X_{i,0}^{+})\\
\quad-\displaystyle\sum_{s \geq 0}\limits\displaystyle\sum_{u=1}^i\limits W^{(1)}_{i,u}t^{-s} W^{(1)}_{u,i+1}t^s-\displaystyle\sum_{s \geq 0}\limits\displaystyle\sum_{u=i+1}^{n}\limits W^{(1)}_{i,u}t^{-s-1} W^{(1)}_{u,i+1}t^{s+1}\\
\qquad\qquad\qquad\qquad\qquad\qquad\qquad\qquad\qquad\qquad\qquad\qquad\qquad\qquad \text{ if $i \neq 0$},
\end{cases}
\end{align*}
\begin{align*}
\Phi^n(X^-_{i,1})&=\begin{cases}
W^{(2)}_{1,n}-2\alpha\Phi^n(X_{0,0}^{-})-\displaystyle\sum_{s \geq 0} \limits\displaystyle\sum_{u=1}^{n}\limits  W^{(1)}_{1,u}t^{-s-1} W^{(1)}_{u,n}t^s,\\
\qquad\qquad\qquad\qquad\qquad\qquad\qquad\qquad\qquad\qquad\qquad\qquad\qquad\qquad \text{ if $i = 0$},\\
 W^{(2)}_{i+1,i}t+\dfrac{i}{2}\Phi^n(X_{i,0}^{-})\\
\quad-\displaystyle\sum_{s \geq 0}\limits\displaystyle\sum_{u=1}^i\limits W^{(1)}_{i+1,u}t^{-s} W^{(1)}_{u,i}t^s-\displaystyle\sum_{s \geq 0}\limits\displaystyle\sum_{u=i+1}^{n}\limits W^{(1)}_{i+1,u}t^{-s-1} W^{(1)}_{u,i}t^{s+1} \\
\qquad\qquad\qquad\qquad\qquad\qquad\qquad\qquad\qquad\qquad\qquad\qquad\qquad\qquad\text{ if $i \neq 0$}.
\end{cases}
\end{align*}
\end{Theorem}
By the definition of $\Psi^n$, we obtain the following theorem.
\begin{Theorem}
Suppose that $\hbar=-1$ and $\ve=-k-(n+1)$. We obtain the following commutative diagram:
\begin{equation*}
\Phi^{n+1}\circ\Psi=\iota\circ\Phi^n.
\end{equation*}
\end{Theorem}
\begin{proof}
The affine Yangian is generated by $X^\pm_{i,0}$ for $0\leq i\leq n-1$ and $X^+_{j,1}$ for $1\leq j\leq n-1$ by the defining relations \eqref{Eq2.1}-\eqref{Eq2.10}. Thus, it is enough to show that
\begin{gather}
\Phi^{n+1}\circ\Psi(X^\pm_{i,0})=\iota\circ\Phi^n(X^\pm_{i,0}),\label{552-1}\\
\Phi^{n+1}\circ\Psi(X^+_{j,1})=\iota\circ\Phi^n(X^+_{j,1}).\label{552-2}
\end{gather}
By the definition of $\Phi^n$, $\iota$ and $\Psi$, \eqref{552-1} holds. By the definition of $\Phi^n$ and $\Psi$, we have
\begin{align*}
&\quad\Phi^{n+1}\circ\Psi(X^+_{j,1})\\
&=\Phi^{n+1}(X^+_{j,1}+\displaystyle\sum_{s \geq 0}\limits E_{j,n+1}t^{-s-1} E_{n+1,j+1}t^{s+1})\\
&=W^{(2)}_{j,j+1}t+\dfrac{i}{2}\Phi^{n+1}(X_{j,0}^{+})\\
&\quad-\displaystyle\sum_{s \geq 0}\limits\displaystyle\sum_{u=1}^j\limits W^{(1)}_{j,u}t^{-s} W^{(1)}_{u,j+1}t^s-\displaystyle\sum_{s \geq 0}\limits\displaystyle\sum_{u=j+1}^{n+1}\limits W^{(1)}_{j,u}t^{-s-1} W^{(1)}_{u,j+1}t^{s+1}\\
&\quad+\displaystyle\sum_{s \geq 0}\limits W^{(1)}_{j,n+1}t^{-s-1} W^{(1)}_{n+1,j+1}t^{s+1}\\
&=W^{(2)}_{j,j+1}t+\dfrac{i}{2}\Phi^{n+1}(X_{j,0}^{+})\\
&\quad-\displaystyle\sum_{s \geq 0}\limits\displaystyle\sum_{u=1}^j\limits W^{(1)}_{j,u}t^{-s} W^{(1)}_{u,j+1}t^s-\displaystyle\sum_{s \geq 0}\limits\displaystyle\sum_{u=j+1}^{n}\limits W^{(1)}_{j,u}t^{-s-1} W^{(1)}_{u,j+1}t^{s+1}.
\end{align*}
On the other hand, by the definition of $\Phi^n$ and $\iota$, we obtain
\begin{align*}
&\quad\iota\circ\Phi^n(X^+_{j,1})\\
&=W^{(2)}_{j,j+1}t+\dfrac{i}{2}\Phi^{n+1}(X_{j,0}^{+})\\
&\quad-\displaystyle\sum_{s \geq 0}\limits\displaystyle\sum_{u=1}^j\limits W^{(1)}_{j,u}t^{-s} W^{(1)}_{u,j+1}t^s-\displaystyle\sum_{s \geq 0}\limits\displaystyle\sum_{u=j+1}^{n}\limits W^{(1)}_{j,u}t^{-s-1} W^{(1)}_{u,j+1}t^{s+1}.
\end{align*}
Thus, the relation \eqref{552-2} holds.
\end{proof}
\section*{Acknowledgement}
The author expresses his sincere thanks to Thomas Creutzig, Nicolas Guay, Shigenori Nakatsuka, Tomoyuki Arakawa, Naoki Genra and Junichi Matsuzawa for the helpful discussion. 
\section*{Data Availability}
The authors confirm that the data supporting the findings of this study are available within the article and its supplementary materials.
\section*{Declarations}
\subsection*{Funding}
This work was supported by JSPS Overseas Research Fellowships, Grant Number JP2360303. 
\subsection*{Conflicts of interests/Competing interests}
The authors have no competing interests to declare that are relevant to the content of this article.
\bibliographystyle{plain}
\bibliography{syuu}
\end{document}